\documentclass[11pt,twoside]{article}
\usepackage{amsmath,amsfonts,amscd,amsthm}

\newtheorem{thm}{Theorem}[section]
\newtheorem{prop}[thm]{Proposition}
\newtheorem{lem}[thm]{Lemma}
\newtheorem{cor}[thm]{Corollary}

\theoremstyle{definition}
\newtheorem{defin}[thm]{Definition}

\theoremstyle{remark}

\numberwithin{equation}{section}

\providecommand\ufootnote[1]{{\let\thefootnote\relax\footnote[0]{#1}}}

\newcommand{\cc}{\mathcal C}
\newcommand{\dc}{\mathcal D}
\newcommand{\ec}{\mathcal E}

\newcommand{\nb}{\mathbb N}

\newcommand{\cb}{\mathbb C}
\newcommand{\zb}{\mathbb Z}

\newcommand{\ci}{{\mathcal C}^\infty}

\newcommand{\ol}{\overline}

\newcommand{\pa}{\partial}
\newcommand{\opa}{\ol\partial}
\newcommand{\wt}{\widetilde}

\DeclareMathOperator{\supp}{supp} \DeclareMathOperator{\im}{Im}
\DeclareMathOperator{\ke}{Ker}

\begin{document}

\title{On the Hausdorff property of some Dolbeault cohomology groups}

\author{Christine LAURENT-THI\'{E}BAUT and Mei-Chi SHAW}

\date{}
\maketitle

\ufootnote{\hskip-0.6cm UJF-Grenoble 1, Institut Fourier, Grenoble,
F-38041, France\newline CNRS UMR 5582, Institut Fourier,
Saint-Martin d'H\`eres, F-38402, France
\newline Department of Mathematics, University of Notre Dame, Notre Dame, IN 46556, USA}

\ufootnote{\hskip-0.6cm  {\it A.M.S. Classification}~: 32C35, 32C37 .\newline {\it Key words}~: Separation, Dolbeault
cohomology, Duality.}

\bibliographystyle{amsplain}


\section{Introduction}
Let $X$ be a complex manifold.  The study of the closed-range property of the Cauchy-Riemann equations
is of fundamental importance both from the sheaf theoretic point of view and the PDE point of view.  In terms of the associated cohomology, it means that the corresponding
cohomology is Hausdorff, hence separated. There are many known results for the Hausdorff property of such cohomologies
in complex manifolds (see in particular \cite{Tr1}, \cite{Tr2} and \cite{Tr3}). For instance, it is well-known that  for a bounded pseudoconvex domain $D$
in $\cb^n$, the Dolbeault cohomology  $H^{p,q}(D)$ in the Fr\'echet space $C^\infty_{p,q}(D)$ vanishes for all $q>0$.
It is also known that the $L^2$ cohomology also vanishes.  Much less is
known about the cohomologies whose topology does not have the
Hausdorff property, even for domains in $\cb^n$ (an example is given
in \cite{Se}, section 14).

In this paper, we study  duality of the Cauchy-Riemann complex in various function spaces and the Hausdorff property
of the corresponding cohomologies.  Such duality is classical  if the complex manifold is compact.  For domains
in a complex manifold with boundary, it has been established   from the Serre duality  between the Fr\'echet spaces and the test forms with compact support  under the inductive limit topology under the closed-range assumption for the Cauchy-Riemann equations.    When the domain has Lipschitz  boundary, an $L^2$ version of the Serre  duality   has been
formulated in \cite{ChaSh}. In this paper we will formulate the duality for the Cauchy-Riemann complex in various function spaces
and use the duality to study the Hausdorff property of Dolbeault cohomology groups.

The plan of the paper is as follows.  In section 2, we  first explain various duality spaces
under  suitable boundary  conditions. In section 3  we use the duality to study the Hausdorff property of the
cohomology groups  for domains with connected complement.
One of the main results in this paper is to show that if moreover the
domain has Lipschitz boundary, the $L^2$ cohomology
$H_{L^2}^{0,1}(D) $ is either 0 (in this case $D$ is pseudoconvex)  or
it is not Hausdorff (see Theorem \ref{thm2}).  In other words, the Cauchy-Riemann equation $\bar\partial$ does not have closed range from $L^2(D)$ to $L^2_{0,1}(D)$ unless $D$ is
pseudoconvex.   This result does not seem to have been observed in the literature (see p. 76
in Folland-Kohn \cite{fk} for previous known examples and some related
results  in Laufer \cite{laufer} and Trapani \cite{Tr1}).

In section 4 we study the duality of cohomologies on annulus type domains. When the domain is the annulus between two   pseudoconvex domains with smooth boundaries,  it is known that the $L^2$ cohomologies are Hausdorff.
This was proved for the annulus between two  strongly pseudoconvex domains in \cite{fk} and between two weakly pseudoconvex domains in \cite{Sh1} and \cite{Sh2}.   But the cohomology groups could be infinite dimensional. When the domain is the annulus between concentric balls, the cohomologies can be expressed
explicitly (see H\"ormander \cite{Hor}). However,  we will show  (see  Corollary \ref{cor4.6})  that if the smoothness assumption is dropped, the cohomologies could
be non-Hausdorff , a  contrast between the annulus between smooth pseudoconvex domains and non-smooth pseudoconvex domains. We also give some results on  sufficient conditions for the Hausdorff property of Dolbeault cohomologies
 of annuli between domains.
 
\bigskip
\noindent
 {\bf Acknowledgements.}  The authors would like to thank Dr. Debraj Chakrabarti  for his   
 comments and suggestions on  the first  version of this paper.  They would also like to thank Professor Stefano Trapani  for 
 pointing out  some of his  earlier work  related to this  paper (see   \cite{Tr1, Tr2, Tr3} in the references). 

\section{Dual complexes}

Let $X$ be an $n$-dimensional complex manifold and $D\subset X$ an
open subset of $X$. We can define on $D$ several spaces of functions: 
\begin{itemize}
\item $\ec(D)$ the space  of $\ci$-smooth functions on $D$ with its classical Fr\'echet topology,

\item $\ci(\ol D)$ the space of the restrictions to $\ol D$ of $\ci$-smooth functions on $X$, i.e. the Whitney space of smooth functions on the closure of $D$, which can be identified with the quotient of the space of $\ci$-smooth functions on $X$ by the ideal of the functions vanishing with all their derivatives on $\ol D$, with the quotient topology, which coincides with the Fr\'echet topology of uniform convergence on $\ol D$ of the function and of all its derivatives,

\item $L^2(D)$ the Hilbert space of the $L^2$-functions on $D$,

\item $\dc(D)$ the space of smooth functions with compact support in $D$ with the usual topology of inductive limit of Fr\'echet spaces,

\item $\dc_{\ol D}(X)$  the subspace of $\dc(X)$ consisting in the functions with support in $\ol D$, endowed with the natural Fr\'echet topology.
\end{itemize}
\begin{defin}
A \emph{cohomological complex of topological vector spaces} is a pair $(E^\bullet,d)$, where $E^\bullet=(E^q)_{q\in\zb}$ is a sequence of locally convex topological vector spaces and $d=(d^q )_{q\in\zb}$ is a sequence of closed linear maps $d^q$ from $E^q$ into $E^{q+1}$ which satisfy $d^{q+1}\circ d^q=0$.

A \emph{homological complex of topological vector spaces} is a pair $(E_\bullet,d)$, where $E_\bullet=(E_q)_{q\in\zb}$ is a sequence of locally convex topological vector spaces and $d=(d_q )_{q\in\zb}$ is a sequence of closed linear maps $d_q$ from $E_{q+1}$ into $E_{q}$ which satisfy $d_q\circ d_{q+1}=0$.
\end{defin}

To any cohomological complex we associate cohomology groups $(H^q(E^\bullet))_{q\in\zb}$ defined by
$$H^q(E^\bullet)=\ker d^q/\im d^{q-1}$$
and endowed with the factor topology and to any homological complex we associate homology groups $(H_q(E_\bullet))_{q\in\zb}$ defined by
$$H_q(E_\bullet)=\ker d_{q-1}/\im d_{q}$$
and endowed with the factor topology.

We will use several cohomological complexes of differential forms associated to the $\opa$-operator and to the previous functions spaces. For some fixed integer $0\leq p\leq n$,
\begin{itemize}
\item let us consider the spaces $\ec^{p,q}(D)$ of $\ci$-smooth $(p,q)$-forms on $D$, set $E^q=0$ and $d^q\equiv 0$, if $q<0$, $E^q=\ec^{p,q}(D)$ and $d_q=\opa$, if $0\leq q\leq n$.

\item let us consider the spaces $\ci_{p,q}(\ol D)$ of $\ci$-smooth $(p,q)$-forms on $\ol D$, set $E^q=0$ and $d^q\equiv 0$, if $q<0$, $E^q=\ci_{p,q}(\ol D)$ and $d^q=\opa$, if $0\leq q\leq n$.

\item let us consider the spaces $L^2_{p,q}(D)$ of $L^2$-forms on $D$, set $E^q=0$ and $d^q\equiv 0$, if $q<0$, $E^q=L^2_{p,q}(D)$ and $d^q=\opa$, the weak maximal realization of $\opa$, i.e. the $\opa$-operator in the sense of currents, if $0\leq q\leq n$. The domain ${\rm Dom}(\opa)$ of $\opa$ is the space of forms in $L^2_{p,q}(D)$ such that $\opa f$ is also in $L^2_{p,q+1}(D)$.

\item let us consider the spaces $L^2_{p,q}(D)$ of $L^2$-forms on $D$, set $E^q=0$ and $d^q\equiv 0$, if $q<0$, $E^q=L^2_{p,q}(D)$ and $d^q=\opa_s$, the strong $L^2$ closure of $\opa$. A form $f\in {\rm Dom}(\opa_s)$ if and only if there exists a sequence $f_\nu\in\ci_{p,q}(X)$ such that $f_\nu\to f$ and $\opa f_\nu\to\opa f$ in $L^2(D)$ strongly, if $0\leq q\leq n$.

\item let us consider the spaces $\dc^{p,q}(D)$ of $\ci$-smooth $(p,q)$-forms with compact support in $D$, set $E^q=0$ and $d^q\equiv 0$, if $q<0$, $E^q=\dc^{p,q}(D)$ and $d^q=\opa$, if $0\leq q\leq n$.

\item let us consider  the spaces $\dc^{p,q}_{\ol D}(X)$ of $\ci$-smooth $(p,q)$-forms on $X$ with support in $\ol D$, set $E^q=0$ and $d^q\equiv 0$, if $q<0$, $E^q=\dc^{p,q}_{\ol D}(X)$ and $d^q=\opa$, if $0\leq q\leq n$.
\end{itemize}

The associated cohomology groups will be denoted respectively by $H^{p,q}_\infty(D)$, $H^{p,q}_\infty(\ol D)$, $H^{p,q}_{L^2}(D)$, $H^{p,q}_{\opa_s,L^2}(D)$, $H^{p,q}_{c,\infty}(D)$ and $H^{p,q}_{c,\infty}(\ol D)$.

\begin{defin}
The \emph{dual complex} of a cohomological complex $(E^\bullet,d)$ of topological vector spaces is the homological complex $(E'_\bullet,d')$,  where $E'_\bullet=(E'_q)_{q\in\zb}$ with $E'_q$ the strong dual of $E^q$ and $d'=(d'_q )_{q\in\zb}$ with $d'_q$ the transpose of the map $d^q$.
\end{defin}

Next we will study   the dual spaces of the spaces of functions we defined at the beginning of the section.
It is well known that the dual space of $\dc(D)$ is the space $\dc'(D)$ of distributions on $D$ and the dual space of $\ec(D)$ is the space $\ec'(D)$ of distributions with compact support in $D$.

Let us consider the space  $\ci(\ol D)$, by definition the restriction map $$ R~:~\ec(X)\to\ci(\ol D)$$  is continuous and surjective, taking the transpose map $^t R$ we get an injection from $(\ci(\ol D))'$ into $\ec'(X)$ the space of distributions with compact support in $X$. More precisely the image of $(\ci(\ol D))'$ by $^t R$ is clearly included in $\ec'_{\ol D}(X)$, the space of distributions on $X$ with support contained in $\ol D$.

Assume $\dc(X\setminus \ol D)$ is dense in the space of $\ci$-smooth functions on $X$ with support contained in $X\setminus D$, which is fulfilled as soon as the boundary of $D$ is sufficiently regular, for example Lipschitz, then any current $T\in \ec'_{\ol D}(X)$ defines a linear form on $\ci(\ol D)$ by setting, for $f\in\ci(\ol D)$, $<T,f>=<T,\wt f>$, where $\wt f$ is a $\ci$-smooth extension of $f$ to $X$ (the density hypothesis implies that $<T,f>$ is independent of the choice of the extension $\wt f$ of $f$), which is continuous by the open mapping theorem.
Consequently, if $D$ has a Lipschitz boundary, the dual space of $\ci(\ol D)$ is the space $\ec'_{\ol D}(X)$ of distributions on $X$ with support contained in $\ol D$.

Again assume the boundary of $D$ is Lipschitz, then $\dc(D)$ is a dense subspace in $\dc_{\ol D}(X)$, the subspace of $\dc(X)$ consisting of  functions with support in $\ol D$, endowed with the natural Fr\'echet topology then the dual space of $\dc_{\ol D}(X)$ will be a subspace of $\dc'(D)$, the space of distribution on $D$. Since $\dc_{\ol D}(X)\subset\dc(X)$ and its topology coincides with the induced topology by the topology of $\dc(X)$, any continuous form on $\dc_{\ol D}(X)$ can be extended as a distribution on $X$. The dual of $\dc_{\ol D}(X)$ is called the space of extendable distribution on $D$ and denoted by $\check\dc'(D)$.

We summarize  the above discussion in the following lemma.
\begin{lem} Let $D$ be  a domain with Lipschitz boundary in a manifold $X$. The dual space of $\ci(\ol D)$ is the space $\ec'_{\ol D}(X)$ of distributions on $X$ with support contained in $\ol D$.
The dual space of $\dc_{\ol D}(X)$ is   the space of extendable distribution on $D$ and denoted by $\check\dc'(D)$.
\end{lem}

The space $L^2(D)$ being an Hilbert space is self-dual and moreover the weak maximal realization of a differential operator and its strong minimal realization are dual to each other (see \cite{ChaSh}).

\medskip
The dual complexes up to a sign of the previous ones are $(E'_\bullet,d')$ with:
\begin{itemize}
\item $E'_q=0$ and $d'_q\equiv 0$, if $q<0$, $E'_q=\ec'^{n-p,n-q}(D)$, the space of currents with compact support in $D$, and $d'_q=\opa$, if $0\leq q\leq n$.

\item $E'_q=0$ and $d'_q\equiv 0$, if $q<0$, and if $D$ has a Lipschitz boundary, $E'_q=\ec'^{n-p,n-q}_{\ol D}(X)$, the space of currents with compact support in $X$ whose support is contained in $\ol D$, and $d'_q=\opa$, if $0\leq q\leq n$.

\item $E'_q=0$ and $d'_q\equiv 0$, if $q<0$,  $E'_q=\L^2_{n-p,n-q}(D)$, the space of $L^2$-forms on $D$, and $d'_q=\opa_c$, the strong minimal realization of $\opa$, if $0\leq q\leq n$. A form $f\in {\rm Dom}(\opa_c)$ if and only if there exists a sequence $f_\nu\in\dc_{n-p,n-q}(D)$ such that $f_\nu\to f$ and $\opa f_\nu\to \opa f$, in $L^2$ strongly.

\item $E'_q=0$ and $d'_q\equiv 0$, if $q<0$, and if $D$ has a rectifiable boundary, $E'_q=\L^2_{n-p,n-q}(X,\ol D)$, the space of $L^2$-forms on $X$ with support in $\ol D$, and $d'_q=\opa_{\wt c}$, the weak minimal realization of $\opa$, i.e. the $\opa$-operator in the sense of currents, if $0\leq q\leq n$. The domain ${\rm Dom}(\opa_{\wt c})$ of $\opa$ is the space of forms in $L^2_{n-p,n-q}(X)$ with support in $\ol D$ such that $\opa f$ is also in $L^2_{n-p,n-q+1}(X)$.

\item $E'_q=0$ and $d'_q\equiv 0$, if $q<0$, $E'_q=\dc'^{n-p,n-q}(D)$, the space of currents on $D$, and $d'_q=\opa$, if $0\leq q\leq n$.

\item $E'_q=0$ and $d'_q\equiv 0$, if $q<0$, $E'_q=\check\dc^{n-p,n-q}(X)$, the space of extendable currents, and $d'_q=\opa$, if $0\leq q\leq n$.
\end{itemize}

The associated homology groups are denoted respectively by $H^{n-p,n-q}_{c,cur}(D)$, $H^{n-p,n-q}_{c,cur}(\ol D)$, $H^{n-p,n-q}_{c,L^2}(D)$, $H^{n-p,n-q}_{{\wt c},L^2}(D)$, $H^{n-p,n-q}_{cur}(D)$ and $\check H^{n-p,n-q}_{cur}(D)$.

\medskip
Let us notice that, if $D$ is a bounded domain with Lipschitz boundary in a complex hermitian manifold $X$ of dimension $n$, it follows from the next lemma that, for $0\leq p\leq n$ and $1\leq q\leq n$, the cohomology groups
$H^{p,q}_{c,L^2}(D)$ and $H^{p,q}_{{\wt c},L^2}(D)$ are isomorphic.

\begin{lem}
Let $D\subset\subset X$ be a relatively compact domain with Lipschitz boundary in a complex hermitian manifold $X$. Then a form $f\in{\rm Dom}(\opa_c)$ if and only if both $f^0$ and $\opa(f^0)$ are in $L^2_*(X)$, where $f^0$ denotes the form obtained by extending the form $f$ by $0$ on $X\setminus D$. We in fact have $(\opa_cf)^0=\opa(f^0)$ in the distribution sense.
\end{lem}

\begin{proof}
By definition, given $f\in{\rm Dom}(\opa_c)$, there is a sequence $(f_\nu)_{\nu\in\nb}$ of smooth forms with compact support in $D$ such that $f_\nu\to f$ and $\opa f_\nu\to\opa_c f$, both in $L^2_*(D)$. Then clearly $(f_\nu)^0\to f^0$ in $L^2_*(X)$. It is also easy to see that $\opa (f_\nu)^0\to\opa f^0$ in the distribution sense in $X$. To see that $\opa f^0=(\opa_c f)^0$, we use integration-by-parts (since $\pa D$ is Lipschitz) to have that for any $\varphi\in\cc^1_*(X)$
\begin{align*}
((\opa_c f)^0,\varphi)_X&=(\opa_c f,\varphi)_D\\
&=\lim_{\nu\to\infty}(\opa f_\nu,\varphi)_D\\
&=\lim_{\nu\to\infty}(f_\nu,\vartheta\varphi)_D\\
&=(f^0,\vartheta\varphi)_X\\
&=(\opa f^0,\varphi)_X
\end{align*}
This proves the "only if " part of the result.
Assume now that both $f^0$ and $\opa f^0$ are in $L^2_*(X)$. To show that $f\in{\rm Dom}(\opa_c)$, we need to construct a sequence $(f_\nu)_{\nu\in\nb}$ of smooth forms with compact support in $D$ which converges in the graph norm corresponding to $\opa$ to $f$. By a partition of unity, this is a local problem near each $z\in\pa D$. By assumption on the boundary, for any point $z\in\pa D$, there is a neighborhood $U$ of $z$ in $X$, and for $\varepsilon\geq 0$, a continuous one parameter family $t_\varepsilon$ of biholomorphic maps from $U$ into $X$ such that $t_\varepsilon(D\cap U)$ is compactly contained in $D$, and $t_\varepsilon$ converges to the identity map on $U$ as $\varepsilon\to 0$. In local coordinates near $z$, the map $t_\varepsilon$ is simply the translation by an amount $\varepsilon$ in the inward normal direction. Then we can approximate $f^0$ locally by $f^{(\varepsilon)}$, where
$$f^{(\varepsilon)}=(t_\varepsilon^{-1})^*f^0$$
is the pullback of $f^0$ by the inverse $t_\varepsilon^{-1}$ of $t_\varepsilon$. A partition of unity argument now gives a form $f^{(\varepsilon)}\in L^2_*(X)$ such that $f^{(\varepsilon)}$ is supported inside $D$ and as $\varepsilon\to 0$,
\begin{align*}
&f^{(\varepsilon)}\to f^0\quad {\rm in}~L^2(X)
&\opa f^{(\varepsilon)}\to\opa f^0\quad {\rm in}~L^2(X).
\end{align*}
Since $\pa D$ is Lipschitz, we can apply Friedrich's lemma(see Lemma 4.3.2 in \cite{ChSh}) to the form  $f^{(\varepsilon)}$ to construct the sequence $(f_\nu)_{\nu\in\nb}\subset\dc(D)$.
\end{proof}

The next proposition is a direct consequence of the Hahn-Banach Theorem
\begin{prop}\label{adh}
Let $(E^\bullet,d)$  and $(E'_\bullet,d')$ be two dual complexes, then
$$\ol{\im d^q}=\{g\in E^{q+1}~|~< g,f>=0, \forall f\in \ke d'_q\}.$$
\end{prop}

\begin{cor}\label{nonsep}
Let $(E^\bullet,d)$  and $(E'_\bullet,d')$ be two dual complexes. Assume $H_{q+1}(E'_\bullet)=0$, then either $H^{q+1}(E^\bullet)=0$ or $H^{q+1}(E^\bullet)$ is not Hausdorff.
\end{cor}
\begin{proof}
Note that
$$\{g\in E^{q+1}~|~< g,f>=0, \forall f\in \ke d'_q\}\subset \ke d^{q+1}.$$
and this inclusion becomes an equality when $H_{q+1}(E'_\bullet)=0$. In that case, it follows from Proposition \ref{adh} that if $d^q$ has closed range then $H^{q+1}(E^\bullet)=0$.
\end{proof}

To end this section let us recall some well-known results about duality for complexes of topological vector spaces proved in \cite{Ca}, and Serre duality proved in \cite{LaLedualite} and \cite{Se}.

\begin{thm}\label{dual}
Let $(E^\bullet,d)$ be a complex of Fr\'echet-Schwarz-spaces or of dual of Fr\'echet-Schwarz-spaces and $(E'_\bullet,d')$ its dual complex. Then for each $q\in\zb$, $H^{q+1}(E^\bullet)$ is Hausdorff if and only if $H_q(E'_\bullet)$ is Hausdorff.
\end{thm}
{\parindent 0pt Theorem \ref{dual} will be used in section \ref{s4} for the complexes $(\ci_{p,\bullet}(\ol D),\opa)$ and $(\dc^{p,\bullet}_{\ol D}(X),\opa)$.}

For the complex $(\ec^{p,\bullet}(D),\opa)$, Serre duality (see
\cite{LaLedualite}, Theorem 1.1) gives

\begin{thm}\label{Sdual}
For all integers $p$,$q$ with $0\leq p\leq n$, $1\leq q\leq n$,
$H^{p,q}_\infty(D)$ is Hausdorff if and only if
$H_{c,\infty}^{n-p,n-q+1}(D)$  is Hausdorff. 
\end{thm}

Finally note that, by the Dolbeault isomorphism, the Dolbeault
cohomology groups $H^{p,q}_\infty(X)$ (resp. $H_{c,\infty}^{p,q}(X)$)
and $H^{p,q}_{cur}(X)$ (resp. $H_{c,cur}^{p,q}(X)$) of a complex
manifold $X$ are isomorphic, so we will simply denote them by 
$H^{p,q}(X)$ (resp. $H_c^{p,q}(X)$).

\section{Hausdorff property for domains with connected complement}\label{s3}

Throughout this section $X$ denotes a non-compact  $n$-dimensional complex manifold and $D\subset\subset X$ a relatively compact subset of $X$ such that $X\setminus D$ is connected.

\begin{lem}\label{supp}
 Assume $X$ satisfies $H^{n,1}_c(X)=0$, then for each current $T\in\ec'^{n,1}(X)$ with support contained in $\ol D$ there exists a $(n,0)$-current $S$ with compact support in $X$, whose support is contained in $\ol D$, such that $\opa S=T$. Moreover if $T\in (L^2_{loc})^{n,1}(X)$ (resp. $T\in\ec^{n,1}(X)$), the solution $S$ is also in $L^2_{loc}(X)$ (resp. $\ec(X)$), hence $H^{n,1}_{\wt c, L^2}(D)=0$ (resp. $H^{n,1}_{c,\infty}(\ol D)=0$) and if the support of $T$ is contained in $D$, the support of $S$ is also contained in $D$, i.e. $H^{n,1}_{c,cur}(D)=0$.
 \end{lem}
 \begin{proof}
 Let $T\in\ec'^{n,1}(X)$ be a current with support contained in $\ol D$. Since $H^{n,1}_c(X)=0$, there  exists a $(n,0)$-current $S$ with compact support in $X$ such that $\opa S=T$. Since the support of $T$ is contained in $\ol D$, the current $S$ is an holomorphic $(n,0)$-form on $X\setminus \ol D$, moreover $S$ has compact support in $X$ and hence vanishes on an open subset of  $X\setminus \ol D$. By  the analytic continuation theorem, the connectedness of $X\setminus D$ implies that the support of $S$ is contained in $\ol D$ and if  moreover the support of $T$ is contained in $D$, the support of $S$ is also contained in $D$. Assume $T\in (L^2_{loc})^{n,1}(X)$ (resp. $\ec^{n,1}(X)$), then the $\opa$-equation has a solution in $L^2_{loc}(X)$ (resp.$\ec(X)$) and as we are in bidegree $(n,1)$, two solutions differ by a holomorphic $(n,0)$-form hence the solution $S$ is in $L^2_{loc}(X)$ (resp. $\ec(X)$).
 \end{proof}

 As a direct consequence of Corollary \ref{nonsep} and Lemma \ref{supp}, we get

 \begin{thm}\label{sep}
 Let $X$ be an $n$-dimensional complex manifold and $D\subset\subset X$ a relatively compact subset of $X$ such that $X\setminus D$ is connected. Assume $X$ satisfies $H^{n,1}_c(X)=0$, then

 (i) Either $H^{0,n-1}(D)=0$ or  $H^{0,n-1}(D)$ is not Hausdorff;

 (ii) If $D$ has a Lipschitz boundary, either $H^{0,n-1}_\infty(\ol D)=0$ or  $H^{0,n-1}_\infty(\ol D)$ is not Hausdorff;

(iii) If $D$ has a Lipschitz boundary, either $H^{0,n-1}_{L^2}(D)=0$ or  $H^{0,n-1}_{L^2}(D)$ is not Hausdorff;

(iv) If $D$ has a rectifiable boundary, either $H^{0,n-1}_{\opa_s,L^2}(D)=0$ or  $H^{0,n-1}_{\opa_s,L^2}(D)$ is not Hausdorff;

 (v) If $D$ has a Lipschitz boundary, either $\check H^{0,n-1}(D)=0$ or  $\check H^{0,n-1}(D)$ is not Hausdorff.
 \end{thm}

 \begin{cor}\label{c2}
 Let  $D$ be a relatively compact open subset of $\cb^2$ such that $\cb^2\setminus D$ is connected, then
 either $D$ is pseudoconvex or $H^{0,1}(D)$ is not Hausdorff. If moreover the boundary of $D$ is Lipschitz, then
  either $D$ is pseudoconvex or $H^{0,1}_{L^2}(D)$ is not Hausdorff.
 \end{cor}
 \begin{proof} The space $\cb^n$ satisfies $H^{n,1}_c(\cb^n)=0$ when $n\ge 2$. From  the characterization of pseudoconvexity for open subset of $\cb^n$ in terms of vanishing of the Dolbeault cohomology,   an open subset $D$ of $\cb^n$ is pseudoconvex if and only if $H^{0,q}(D)=0$ for all $1\leq q\leq n-1$.
 Thus the first part of the theorem is a consequence of Theorem \ref{sep}.

 On the other hand, if   $D\subset\subset \cb^n$ is bounded pseudoconvex, then $H^{0,q}_{L^2}(D)=0$ for all $1\leq q\leq n-1$ by H\"ormander $L^2$-theory. The converse is also true provided  $D$ has Lipschitz boundary or more generally,
 $D$ satisfies $\text{interior}(\ol D)=D$ (see e.g. the remark at the end of the paper in \cite{Fu}).
  \end{proof}

Note that the first assertion of Corollary \ref{c2} was already proved by
Trapani (\cite{Tr1}, Theorem 2), where a characterization of Stein
domains in a Stein manifold of complex dimension $2$ is given.

  \begin{thm}\label{thm2}
 Let  $D$ be a relatively compact open subset of $\cb^2$ such that $\cb^2\setminus D$ is connected. Suppose $D$
 is not pseudoconvex.   Then  $\bar\partial ~:~ \ci(D)\to \ci_{0,1}(D)$ does not have closed range.   If moreover the boundary of $D$ is Lipschitz, then  $\bar\partial ~:~ L^2(D)\to L^2_{0,1}(D)$ does not have closed range either.
 \end{thm}
Next we will compare the cohomologies with smooth data up to the boundary.
For pseudoconvex open subsets of $\cb^2$ with Lipschitz boundary, Theorem \ref{sep} implies that we have either $H^{0,1}_\infty(\ol D)=0$ or $H^{0,1}_\infty(\ol D)$ is not separated. By the classical Kohn's theorem (\cite{Ko}) for the $\opa$-problem on pseudoconvex domains with $\ci$-smooth boundary, we have that $H^{0,1}_\infty(\ol D)=0$, when moreover $D$ has $\ci$-smooth boundary and by Dufresnoy's result (\cite{Du}) on the $\opa$-problem for differentiable forms in the sense of Whitney, we have that $H^{0,1}_\infty(\ol D)=0$, when moreover $\ol D$ admits a
sufficiently nice  Stein neighborhood basis.

 Next we will show that there exists an example of pseudoconvex domain $D$  in $\cb^2$ such that
 the
cohomology group $H^{0,1}_\infty(\ol D)$ is infinite dimensional.
Let $T$ be  the Hartogs triangle   in $\cb^2$
$$T=\{(z,w)\in\cb^2~|~|z|<|w|<1\}$$
which is pseudoconvex, hence $ H^{0,1}_\infty(T)= 0$.

It follows from a paper by J. Chaumat and A.-M. Chollet
\cite{ChChHartogs} that for any $\zeta$ in the bidisc
$P=\Delta\times\Delta$ and $\zeta\in P\setminus \ol T$,  their exists
a $\ci$-smooth, $\opa$-closed $(0,1)$-form $\alpha_\zeta$ defined in
$\cb^2\setminus \{\zeta\}$ such that there does not exist any
$\ci$-smooth function $\beta$ on $\ol T$ such that $\opa
\beta=\alpha_\zeta$. In particular the $\opa$-equation $\opa
u=\alpha_\zeta$ cannot be solved in the $\ci$-smooth category in any
neighborhood of $\ol T$. Since by an argument due to Laufer \cite{lauf},
the Dolbeault cohomology group $H^{0,1}_\infty(\ol T)$ is either zero
or infinite dimensionnal, we can conclude 
\begin{prop}The
cohomology group $H^{0,1}_\infty(\ol T)$ is infinite dimensional.
\end{prop}
 As the boundary of $T$ is not Lipschitz we cannot apply Theorem \ref{sep} and we do not know if this group is Hausdorff or not.
 But if we consider the strong $L^2(D)$-cohomology then either $H^{0,1}_{\opa_s,L^2}(T)=0$ or  $H^{0,1}_{\opa_s,L^2}(T)$ is not Hausdorff since the boundary of $T$ is rectifiable.

\section{Cohomology in an annulus}\label{s4}

Let $X$ be a non-compact  $n$-dimensional complex manifold and $D\subset\subset   X$ be a relatively compact subset  of $X$ such that $X\setminus D$ is connected. We will first study the relations between the Dolbeault cohomology groups of $D$ and some other Dolbeault cohomology groups of $X\setminus \ol D$.

\begin{prop}\label{anneau-ouvert}
Let $X$ be a non-compact complex manifold of complex dimension $n\geq 2$ and $D$ be a relatively compact domain in $X$ such that $X\setminus D$ is connected. Assume $H^{n,1}_c(X)=0$ and $H^{0,n-1}(X)=0$. Then if $H_{c,\infty}^{0,n}(X\setminus\ol D)$ is Hausdorff, for any neighborhood $U_{\ol D}$ of $\ol D$ such that $X\setminus U_{\ol D}$ is connected, there exists a neighborhood $V_{\ol D}$ of $\ol D$ such that $V_{\ol D}\subset\subset U_{\ol D}$ and for each $f\in \ec^{0,n-1}(U_{\ol D})$ such  that $\opa f=0$, there exists $u\in \ec^{0,n-2}(V_{\ol D})$ such  that $\opa u=f$ on $V_{\ol D}$.
\end{prop}

\begin{proof}
 Let $W_{\ol D}\subset\subset U_{\ol D}$ be a neighborhood of $\ol D$ and $\chi$ be a $\ci$-smooth function on $X$ with compact support in $U_{\ol D}$ and constant equal to $1$ on $W_{\ol D}$ and $f\in \ec^{0,n-1}(U_{\ol D})$ such  that $\opa f=0$.  Set $g=\opa(\chi f)$. The form $g$ is a $\opa$-closed $(0,n)$-form with compact support in $X\setminus\ol D$.

 First we want to prove that $g$  belongs to the closure of the image
 by $\opa$ of the $\ci$-smooth $(0,n-1)$-forms with compact support in
 $X\setminus\ol D$. By Proposition \ref{adh} and by the regularity of
 the $\opa$-operator in complex manifolds, it suffices  to prove that
 $\int_{X\setminus\ol D} g\wedge\theta =0$ for all holomorphic
 $(n,0)$-forms $\theta\in\ec^{n,0}(X\setminus\ol D)$. Since $X$ is not
 compact, $X\setminus D$ is connected and $H^{n,1}_c(X)=0$, by the Hartogs extension phenomenon, the holomorphic $(n,0)$-form $\theta$ extends to $X$ in a holomorphic $(n,0)$-form $\wt\theta$ and
 $$\int_{X\setminus\ol D} g\wedge\theta =\int_{X\setminus\ol D} \opa(\chi f)\wedge\wt\theta=\int_{X} \chi f\wedge\opa\wt\theta=0.$$

Then it follows from the Hausdorff property of the cohomological group
$H_c^{0,n}(X\setminus\ol D)$ that their exists a $(0,n-1)$-form $v$ of
class $\ci$ with compact support in $X\setminus\ol D$ such that $\opa
v=g$. Note that by Lemma 2.3 in \cite{LaLedualite} the support of $v$ depends only of the support of $g$ which in the present setting is only related to the choice of the function $\chi$. Consider now the $(0,n-1)$-form $\chi f-v$. After extension of $v$ by $0$ in $\ol D$ it is defined on the whole $X$, moreover $\opa (\chi f-v)=0$ on $X\setminus\ol D$ by definition of $v$ and  $\opa (\chi f-v)=0$ on a neighborhood of $\ol D$ since $\supp v\subset X\setminus\ol D$ and $\opa(\chi f)=\opa f=0$ on $W_{\ol D}$. Using that $H^{0,n-1}(X)=0$, we get a $\ci$-smooth $(0,n-2)$-form $h$ on $X$ such that $\chi f-v=\opa h$ in $X$, in particular $f=\opa h$ on $V_{\ol D}=W_{\ol D}\cap (X\setminus \supp v)$.
\end{proof}

Let us study the converse of Proposition \ref{anneau-ouvert}.

\begin{prop}\label{recip}
Let $X$ be a complex manifold of complex dimension $n$ such that $H^{0,n}_c(X)$ is Hausdorff and $D$ be a relatively compact domain in $X$. If for any neighborhood $U_{\ol D}$ of $\ol D$ there exists a neighborhood $V_{\ol D}$ of $\ol D$ such that $V_{\ol D}\subset\subset U_{\ol D}$ and for each $f\in \ec^{0,n-1}(U_{\ol D})$ such  that $\opa f=0$, there exists $u\in \ec^{0,n-2}(V_{\ol D})$ such  that $\opa u=f$ on $V_{\ol D}$, then $H_{c,\infty}^{0,n}(X\setminus\ol D)$ is Hausdorff.
\end{prop}

\begin{proof}
By Theorem 2.7 in \cite{LaLedualite} it is sufficient to prove that for each compact set $K\subset X\setminus\ol D$ the space $\dc_K^{0,n}(X\setminus\ol D)\cap\opa\dc^{0,n-1}(X\setminus\ol D)$ is topologically closed in the space $\dc^{0,n}(X\setminus\ol D)$, where $\dc_K^{0,n}(X\setminus\ol D)$ denotes the space of $\ci$-smooth $(0,n)$-forms on $X\setminus\ol D$ with support in $K$. Let $K$ be a fixed compact subset of $X\setminus\ol D$.

First we will prove that
$$\dc_K^{0,n}(X\setminus\ol D)\cap\opa\dc^{0,n-1}(X\setminus\ol D)=\dc_K^{0,n}(X\setminus\ol D)\cap\opa\dc^{0,n-1}(X).$$
It is clear that
$$\dc_K^{0,n}(X\setminus\ol D)\cap\opa\dc^{0,n-1}(X\setminus\ol D)\subset\dc_K^{0,n}(X\setminus\ol D)\cap\opa\dc^{0,n-1}(X).$$
For the converse inclusion let $f\in\dc_K^{0,n}(X\setminus\ol D)\cap\opa\dc^{0,n-1}(X)$, then $f=\opa g$ with $g\in \dc^{0,n-1}(X)$. Since $\supp f\subset K$, the form $g$ is $\opa$-closed on some neighborhood $U_{\ol D}$ of $\ol D$. From the hypothesis we get that there exists a neighborhood $V_{\ol D}$ of $\ol D$ such that $V_{\ol D}\subset\subset U_{\ol D}$ and a form $h\in \ec^{0,n-2}(V_{\ol D})$ such  that $\opa h=g$ on $V_{\ol D}$. Choose $\chi$ a $\ci$-smooth function on $X$ with compact support in $V_{\ol D}$ and such that $\chi=1$ on a neighborhood of $\ol D$, then $f=\opa(g-\opa(\chi h))$ and $\supp(g-\opa(\chi h))$ is a compact subset of $X\setminus\ol D$.

The Hausdorff property of the cohomological group $H^{0,n}_c(X)$ implies that for each compact set $K\subset X$ the space $\dc_K^{0,n}(X)\cap\opa\dc^{0,n-1}(X)$ is topologically closed in the space $\dc^{0,n}(X)$. Since $K\subset X\setminus\ol D$ we get that $\dc_K^{0,n}(X\setminus\ol D)\cap\opa\dc^{0,n-1}(X)$ is topologically closed in the space $\dc^{0,n}(X\setminus\ol D)$, which ends the proof of the proposition.
\end{proof}

Note that under the hypothesis of Proposition \ref{recip}, we get that if $\ol D$ admits a Stein neighborhood basis then $H_{c,\infty}^{0,n}(X\setminus\ol D)$ is Hausdorff (A slightly stronger result  is proved in section 4 of \cite{LaLedualite}).

\begin{cor}\label{cor4.3}
 Let $X$ be a Stein manifold of complex dimension $n\geq 2$ and $D$ be a relatively compact domain in $X$ such that $X\setminus D$ is connected. Then $H^{n,1}_\infty(X\setminus\ol D)$ is Hausdorff if and only if for any neighborhood $U_{\ol D}$ of $\ol D$ such that $X\setminus U_{\ol D}$ is connected, there exists a neighborhood $V_{\ol D}$ of $\ol D$ such that $V_{\ol D}\subset\subset U_{\ol D}$ and for each $f\in \ec^{0,n-1}(U_{\ol D})$ such  that $\opa f=0$, there exists $u\in \ec^{0,n-2}(V_{\ol D})$ such  that $\opa u=f$
 \end{cor}
 \begin{proof}
 If  $X$ is a Stein manifold of complex dimension $n\geq 2$ and $D$ be a relatively compact domain in $X$ such that $X\setminus D$ is connected, the hypotheses of Proposition \ref{anneau-ouvert} and  Proposition \ref{recip} are fulfilled  and moreover, by the Serre duality in $X\setminus\ol D$, $H_{c,\infty}^{0,n}(X\setminus\ol D)$ is Hausdorff if and only if $H^{n,1}_\infty(X\setminus\ol D)$ is Hausdorff.
 \end{proof}

Corollary \ref{cor4.3} was already proved in \cite{Tr1} as a corollary
to a
result related to Propositions \ref{anneau-ouvert} and \ref{recip} (see
\cite{Tr1}, Lemmas 8 and 9).
\medskip

In the same setting, we are going to consider now the case of the cohomology of closed subsets.

\begin{prop}\label{anneau-ferme}
Let $X$ be a non compact complex manifold of complex dimension $n\geq 2$ and $D$ be a relatively compact domain in $X$ with Lipschitz boundary and such that $X\setminus D$ is connected. Assume $H^{n,1}_c(X)=0$, $H^{0,n-1}(X)=0$ and $H^{0,n}_c(X)$ are Hausdorff, then $H_{c,\infty}^{0,n}(X\setminus D)$ is Hausdorff if and only if $H^{0,n-1}_\infty(\ol D)=0$.
\end{prop}
\begin{proof}
Assume $H_{c,\infty}^{0,n}(X\setminus D)$ is Hausdorff. Let
$f\in\ci_{0,n-1}(\ol D)$ be a $\opa$-closed form, then if $\wt f$ is a
$\ci$-smooth extension with compact support of $f$ to $X$, the
$(0,n)$-form $\opa\wt f$ has compact support in $X\setminus D$ and
satisfies $<T,\opa\wt f>=0$ for all
$T\in\check\dc'^{n,0}(X\setminus\ol D)$ with $\opa T=0$, in fact $T$
is a holomorphic $(n,0)$-form on $X\setminus\ol D$ which extends to
$X$ in a  holomorphic $(n,0)$-form $\theta$ by the Hartogs extension
phenomenon since $X$ is not compact, $X\setminus D$ is connected and $H^{n,1}_c(X)=0$, hence
$$<T,\opa\wt f>=\int_X\theta\wedge\opa\wt f=\int_X\opa\theta\wedge\wt f=0.$$

By the Hausdorff property of the group $H_{c,\infty}^{0,n}(X\setminus D)$, there exists a $(0,n-1)$-form $g$ with compact support in $X\setminus D$ such that
$\opa\wt f=\opa g$. Consequently the form $\wt f-g$ is a $\opa$-closed $(0,n-1)$-form on $X$ whose restriction to  $\ol D$ is equal to $f$. As $H^{0,n-1}(X)=0$, we get $\wt f-g=\opa h$ for some $\ci$-smooth form $h$ on $X$ and hence $f=\opa h$ on $D$, which proves that $H^{0,n-1}_\infty(\ol D)=0$.

Let us prove the converse. Let $f\in\ci_{0,n}(X)$ be a form with compact support in $X\setminus D$ orthogonal to the $\ci$-smooth $(n,0)$-forms on $X\setminus D$, which are holomorphic on $X\setminus \ol D$. Then, $f$ is orthogonal to the holomorphic $(n,0)$-forms in $X$ and the Hausdorff property of $H^{0,n}_c(X)$ implies that there exists a $\ci$-smooth $(0,n-1)$-form $g$ with compact support in $X$ such that $f=\opa g$. The support property of $f$ implies that $g$ restricted to $\ol D$   is a $\opa$-closed $\ci$-smooth $(0,n-1)$-form on $\ol D$. Since    $H^{0,n-1}_\infty(\ol D)=0$, we get $g=\opa h$ for some $\ci$-smooth $(0,n-2)$-form on $\ol D$. Let $\wt h$ be a $\ci$-smooth extension with compact support of $h$ to $X$, then $g-\opa\wt h$ has compact support in $X\setminus D$ and satisfies $\opa (g-\opa\wt h)=f$, which implies that $H_{c,\infty}^{0,n}(X\setminus D)$ is Hausdorff.
\end{proof}

Note that Proposition \ref{anneau-ferme} also  holds for extendable currents and  the proof follows the same lines as the previous one.
Using duality on $X\setminus D$ and Theorem \ref{sep} we get
\begin{cor}
Let $X$ be a Stein manifold of complex dimension $n\geq 2$ and $D$ be a relatively compact domain in $X$ with Lipschitz boundary and such that $X\setminus D$ is connected, then

(i) $\check H^{n,1}(X\setminus D)$ is Hausdorff if and only if $H^{0,n-1}_\infty(\ol D)$ is Hausdorff;

(ii) $H^{n,1}_\infty (X\setminus D)$ is Hausdorff if and only if $\check H^{0,n-1}(\ol D)$ is Hausdorff.
\end{cor}

 Let us consider the special example where $X=B\subset\cb^2$, a ball of radius $R\geq 2$ in $\cb^2$, then $B$ is a Stein manifold of dimension $2$ and
$D=\Delta\times\Delta$ is the bidisc, then $H^{0,1}_\infty(\ol {\Delta\times\Delta})=0$, hence $\check H^{2,1}(B\setminus (\Delta\times\Delta))$ is Hausdorff.
Similarly, we also have from Corollary \ref{cor4.3} that   $H^{2,1}(B\setminus\ol{\Delta\times\Delta})$ is Hausdorff.

On the other hand, if we consider  the Hartogs triangle $T\subset\subset  X=B\subset\cb^2$,
then using  the result of \cite{ChChHartogs} (see the end of section \ref{s3} in this paper) and Proposition \ref{anneau-ouvert}, we get

\begin{cor}\label{cor4.6}
Let $B\subset\cb^2$ be a ball of radius $R\geq 2$ in $\cb^2$ and $T$ the Hartogs triangle, then both cohomological groups $H^{0,2}_c(B\setminus \ol T)$ and $H^{2,1}(B\setminus \ol T)$ are not Hausdorff.
\end{cor}
We remark that the fact that $H^{2,1}(B\setminus \ol T)$ is  not Hausdorff follows from the   Serre duality.
\medskip

\medskip

Now  we will extend partially Proposition \ref{anneau-ferme} to the $L^2$ setting. Let $W^1(D)$ be the Sobolev space
and   denote by $H^{p,q}_{W^1}(D)$ the associated cohomology groups.

\begin{prop}\label{anneau-fermeL2}
Let $X$ be a Stein  manifold of complex dimension $n\geq 2$ and $D$ be a relatively compact domain in $X$ with Lipschitz  boundary and such that $X\setminus D$ is connected.
 Then  $H_{c,L^2}^{0,n}(X\setminus\ol D)$ is Hausdorff if and only if  $H^{0,n-1}_{W^1}(D)=0$.
\end{prop}
\begin{proof}
Let $f\in W^1_{0,n-1}(D)$ be a $\opa$-closed form on $D$, and let $\wt f$ be a  $W^1$ extension with compact support of $f$ to $X$. This is possible since the boundary of $D$ is Lipschitz.  The $(0,n)$-form $\opa\wt f$ has $L^2$ coefficients and compact support in $X\setminus D$. Furthermore, it    satisfies
$$\int_{X\setminus D}\theta\wedge\opa\wt f=0$$ for all holomorphic $(n,0)$-form on $X\setminus\ol D$.  In fact, since $X\setminus D$ is connected, by the Hartogs extension phenomenon,  $\theta$ extends to $X$ in a  holomorphic $(n,0)$-form $\wt\theta$ and
$$\int_{X\setminus D}\theta\wedge\opa\wt f=\int_X\wt\theta\wedge\opa\wt f=\int_X\opa\wt\theta\wedge\wt f=0.$$

By the Hausdorff property of the group $H_{c,L^2}^{0,n}(X\setminus\ol D)=H_{\wt c,L^2}^{0,n}(X\setminus\ol D)$, there exists a form $g\in L^2_{(0,n-1)}(X)$ with compact support in $X\setminus D$ such that
$\opa\wt f=\opa g$. Consequently the form $\wt f-g$ is a $\opa$-closed $(0,n-1)$-form on $X$ whose restriction to  $D$ is equal to $f$. As $X$ is Stein,  $H_{L^2_{\text{loc}}}^{0,n-1}(X)=0$. Thus  we get $\wt f-g=\opa h$ for some $L^2_{\text{loc}}$ form $h$ on $X$. It follows from the interior regularity for $\opa$, we can have $h\in W^1(D)$ and hence $f=\opa h$ on $D$, which proves that $H^{0,n-1}_{W^1}(D)=0$.

Conversely let $f\in L^2_{0,n}(X)$ with compact support in $X\setminus D$, orthogonal to the $\opa$-closed $(n,0)$-forms $L^2$ in $X\setminus D$ and in particular to the holomorphic $(n,0)$-forms in $X$. The Hausdorff property of $H^{0,n}_{c,L^2}(X)$ implies that there exists a $(0,n-1)$-form $g\in L^2_{0,n-1}(X)$ with compact support in $X$ such that $f=\opa g$. Using the interior regularity, we have $g$ has $W^1$ coefficients on $D$. Since the support of $f$ is contained in $X\setminus D$, $g$ is $\opa$-closed in $D$ and as $H^{0,n-1}_{W^1}(D)=0$, we get $g=\opa h$ for some $(0,n-2)$-form $h$ in $W^1_{0,n-2}(D)$. Let $\wt h$ be a $W^1$ extension of $h$ with compact support in  $X$, then $g-\opa\wt h$ vanishes on $D$ and satisfies $\opa (g-\opa\wt h)=f$. This shows that $H_{c,L^2}^{0,n}(X\setminus\ol D)$ is Hausdorff.

\end{proof}

Using the $L^2$-duality between $\opa$ and $\opa_{c}$,  we get
\begin{cor}
Let $X$ be a bounded pseudoconvex domain in $\cb^n$,  $n\geq 2$ and $D$ be a relatively compact domain in $X$ with Lipschitz  boundary and such that $X\setminus D$ is connected, then either $H^{0,n-1}_{W^1}(D)=0$ or $H^{n,1}_{L^2} (X\setminus D)$ is not Hausdorff.
\end{cor}

We remark that when $D$ is a pseudoconvex domain with $\cc^2$ boundary, we can further obtain  the duality
between $L^2$ cohomologies  $H_{L^2}^{n,n-1}(X\setminus\ol D)$  and the Bergman space  $H^{0,0}_{L^2}(D)$ (see
\cite{Hor},
\cite{Sh1}, \cite{Sh2}).   However, not much is known when $D$ has only Lipschitz boundary.

\providecommand{\bysame}{\leavevmode\hbox to3em{\hrulefill}\thinspace}
\providecommand{\MR}{\relax\ifhmode\unskip\space\fi MR }
\providecommand{\MRhref}[2]{%
  \href{http://www.ams.org/mathscinet-getitem?mr=#1}{#2}
}
\providecommand{\href}[2]{#2}

\enddocument

\end